\documentclass[11pt,a4paper]{article}
\usepackage{indentfirst}
\usepackage{fancyhdr}
\usepackage{color}
\usepackage{mathrsfs}
\usepackage{amsfonts,amssymb,amsmath,amsthm}
\usepackage{authblk}
\usepackage[unicode,pdftex]{hyperref}
\hypersetup{
	colorlinks = true,
	citecolor = green,
	anchorcolor = blue,
	linkcolor = blue}

\newtheorem{thm}{Theorem}[section]
\newtheorem{lemma}[thm]{Lemma}
\newtheorem{coro}[thm]{Corollary}
\newtheorem{prop}[thm]{Proposition}

\newtheoremstyle{rem}{10pt}{10pt}{\rmfamily}{}{\bfseries}{.}{.5em}{}
\theoremstyle{rem}
\newtheorem{rem}[thm]{Remark}

\numberwithin{equation}{section} 

\title{On the global well-posedness of stochastic Schr\"{o}dinger–Korteweg-de Vries system}
\author[1,2]{Jie Chen}
\author[1,*]{Fan Gu}
\author[1]{Boling Guo}
\affil[1]{\scriptsize \textit{Institute of Applied Physics and Computational Mathematics, Beijing 100088, P.R. China}}
\affil[2]{\scriptsize \textit{School of Sciences, Jimei University, Xiamen 361021, P.R. China}}
\date{}

\begin{document}
	\maketitle
	
	\begin{abstract}
		In this paper, we study the global well-posedness of the stochastic S-KdV system in $H^1(\mathbb{R})\times H^1(\mathbb{R})$, which are driven by additive noises. It is difficult to show the global well-posedness of a related perturbation system even for smooth datum and stochastic forces. To overcome it, we introduce a new sequence of approximation equations, which is the key of this paper. We establish priori estimates, global well-posedness and convergences of these approximation equations, which help us to get a pathwise priori estimate of initial system.
	\end{abstract}
	
	\section{Introduction}\label{chapterintroduction}
	In this paper, we study the global well-posedness of the stochastic S–KdV system. The corresponding deterministic model is 
	\begin{equation*}\label{model}\tag{S-KdV}
		\left\{
		\begin{aligned}
			&i\partial_t u+\partial_{xx}u = \gamma_1 uv+\beta|u|^2u,\\
			&\partial_t v+\partial_{xxx}v = \gamma_2\partial_x(|u|^2)-v\partial_x v,\\
			&(u,v)|_{t = 0} = (u_0,v_0),
		\end{aligned}
		\right.
	\end{equation*}
	where $\gamma_1, \gamma_2,\beta$ are real-valued constants, $u$ is complex-valued and $v$ is real-valued.

	The deterministic \eqref{model} is an important model in fluid mechanics and plasma physics. It is devoted to describing the interactions between short waves $u(x,t)$ and long waves $v(x,t)$. The case $\beta=0$ appears in the study of resonant
	interaction between short and long capillary-gravity waves on water of a uniform finite depth, in plasma physics and in a diatomic lattice system (For more details, one can see \cite{77dynamics} \cite{77theory} \cite{corcho2007well}).  The well-posedness of this coupled system is widely researched. In \cite{Guo}, Guo studied the Cauchy problem with  in $L([0,T];H^s(\mathbb{R})\times H^s(\mathbb{R}))$, for all integers $s \geq 3$. In \cite{H1H1}, Guo-Miao proved the global well-posedness in  $H^s(\mathbb{R})\times H^s(\mathbb{R})$, for $s\in\mathbb{N^+}, \beta=0$. Concerning the well-posedness in $H^s(\mathbb{R})\times H^s(\mathbb{R})$, Corcho-Linares \cite{corcho2007well} obtained the local well-posedness for \eqref{model} when $s\geq 1/2$. See also \cite{guo2010well} for the local well-posedness of \eqref{model} in $L^2(\mathbb{R})\times H^{-3/4}(\mathbb{R})$. 
	
	As far as we know, there is no paper considering the well-posedness of \eqref{model} perturbed by stochastic forces. For single stochastic dispersive equations, 
	\cite{KdVH1} studies the strong solution of stochastic  KdV equation in Strichartz space $\tilde{X}_\sigma(T)\cap L_T^{\infty}H^1_x(\mathbb{R})$, $\sigma \in ({3}/{4},1)$ with initial value in $L^2_{\omega}H^1_x(\mathbb{R})\cap L_\omega^4L^2_x$. By using Bourgain spaces, de Bouard-Debussche-Tsutsumi \cite{de1999white} proved the well-posedness of the strong solution of stochastic KdV equation in $L^2(\mathbb{R})$. There are also a lot of references about the well-posedness of stochastic Schr\"{o}dinger equation. See for example \cite{de1999stochastic}, \cite{schrodingerH1}, \cite{zhangd}.
	
	In this paper, we study the well-posedness of the strong solutions for the stochastic S-KdV system driven by additive noises in $H^1(\mathbb{R})\times H^1(\mathbb{R})$. We focus on the global well-posedness when $\gamma_1\cdot \gamma_2>0$.
	
	Let
	$$w = (u,v),\quad L(w) =(L_1(w),L_2(w))= (i\partial_{x}^2u,-\partial_{x}^3v)$$
	and
	$$N(w) =(N_1(w),N_2(w))= (i\gamma_1 uv+i\beta|u|^2u, \gamma_2\partial_x(|u|^2)-v\partial_x v).$$
	We fix a probability space $\left(\Omega, \mathscr{F},\mathbb{P},\left(\mathscr{F}_t\right)_{t\in\left[0,T\right]} \right)$. $\frac{dW^{(1)}}{dt}, \frac{dW^{(2)}}{dt}$ are two independent white noises on $L^2\left( \mathbb{R}\right)$ adapted to $\left\{\mathscr{F}_t\right\}_{t\in\left[0,T\right]}$. $W^{(k)}, k=1,2$ can be represented as $\sum_{i=0}^{+\infty}\beta_i^{(k)}(t)e_i^{(k)}$, where $\{\beta_i^{(k)}\}$ is a sequence of mutually independent real standard Brownian motions and $\{e_i^{(k)}\}$ is an orthonormal basis of $L^2(\mathbb{R})$.
	
	Let us write the stochastic S-KdV equation in the It\^{o} form:
	\begin{equation}\label{smodel}
		\left\{
		\begin{aligned}
			&dw = (L(w)+N(w))dt+(\Phi^{(1)} dW_t^{(1)},\Phi^{(2)}dW^{(2)}_t),\\
			&w|_{t = 0} = w_0.
		\end{aligned}
		\right.
	\end{equation}
	Here $\Phi=(\Phi^{(1)},\Phi^{(2)})$ is a linear operator from $\left(L^2(\mathbb{R}), L^2(\mathbb{R})\right) $ to $(H_1, H_2)$ where $H_1, H_2$ are two Hilbert spaces. 
	
	Usually, to show the well-posedness of \eqref{smodel}, one needs the following steps: Firstly, we prove that the strong solution of the linear stochastic equation are almost surely in the aiming space. Secondly, by a pathwise fixed point argument,  we get a local mild solution in $[0,T(\omega)]$ for \eqref{smodel}. Finally, by a priori estimate, we extend the solution to $[0,T]$ for any $T>0$. 
	
	In our work, problems arise in the last step. To obtain the priori estimate, we need to use the conservation laws of \eqref{model}. However, for \eqref{smodel}, in the priori estimate of the conserved quantities, we need higher regularity than the solution has, to explain the calculation in the strong sense. Thus, we have to consider a sequence of approximation equations with high regularity and prove a uniform priori estimate of this sequence.
	
	For deterministic S-KdV in $H_x^1\times H_x^1$, to get a priori estimate of it, we only need approximation equations with smooth initial values.
	For single stochastic dispersion equations (for example \cite{KdVH1}), to get a priori estimate, it is enough to smooth the initial data and the noise term. But in our situation, it is hard to prove a  priori estimate even for smooth  initial datum and noises. Thus, we need construct new suitable approximation equations. There are two principles of the construction: One is that they must be global well-posed in the high regular space. The other is that they must have enough conservation laws to get a uniform priori estimate in $H^1(\mathbb{R})\times H^1(\mathbb{R})$. These are the key points of whole paper. 
	
	Our main result is the following theorem:
	\begin{thm}\label{thm:main}
		Suppose $u_0,v_0$ are $\mathscr{F}_0$-measurable, $ \Phi\in L_2^{0,1} \times L_2^{0,1},\  \gamma_1\cdot \gamma_2>0$ and
		$$u_0 \in L^2(\Omega; H^{1}(\mathbb{R})) \cap L^4(\Omega;L^4(\mathbb{R}))\cap L^{10}(\Omega;L^2(\mathbb{R})),
		$$
		$$
		\ v_0 \in L^2(\Omega; H^{1}(\mathbb{R})) \cap L^{3}(\Omega;L^3(\mathbb{R})). $$ 
		Then, for any $ T>0$, \eqref{smodel} exists a unique strong solution $w \in {X_{1}(T)} \ a.s.\mathbb{P}$.
	\end{thm}

	Our paper is organized in the following manner: In Section \ref{chapterpreliminary}, we introduce notations, definitions, the workspace and properties of  the linear stochastic equation. In Section \ref{chapterpriori}, we construct the approximation equations and prove a uniform priori estimate. In Section \ref{chapterwellposedness}, we prove the global well-posedness of the approximation equations. Finally, in Section \ref{chapterconvergence}, we get convergences of the approximation equations in $X_1(T)$ and finish the proof of Theorem \ref{thm:main}.

	\section{Preliminary}\label{chapterpreliminary}
	
	In this section,  we  give some notations  and definitions. We will propose the workspace of our paper and introduce  necessary estimates of the linear stochastic equation. 
	
	In our paper, $C, \tilde{C}, C_k,... $ denote various constants which may depend on $\gamma_1,\gamma_2,\beta$ only. We use $C(x,y,\cdots)$ to represent constants depending on some parameters $x,y\cdots$. For $a, b \in \mathbb{R}^+$, $ a\lesssim b$ means that there exists $C>0$ such that $a \leq Cb$. We use $\text{supp}~f$ to denote the supporting set of $f$.

	For $u \in \mathscr{S}'(\mathbb{R})$, we use $\mathcal{F}u$ and $ \hat{u}$ to denote the Fourier transform of $u$. We denote   $Du=\mathcal{F}^{-1}(|\xi|\hat{u}(\xi)) $ and $J^su=\mathcal{F}^{-1}\left((1+\xi^2)^{\frac{s}{2}}\hat{u}(\xi)\right)$.
	We denote by $\left(\cdot,\cdot\right)$ the $L^2$ inner product
	$$\left(f(x),g(x)\right)=\int_{\mathbb{R}^2} f(x) \overline{g(x)} dx.$$
	For $u$, we use complex-valued function spaces, and for $v$, we use real-valued function spaces. For $s\in \mathbb{R}$, with a little abuse of notation, we use $H^s(\mathbb{R})$ to denote the Sobolev space of order $s$ and $\mathcal{H}^s_x:=H^s_x\times H^s_x$. 
	We also use $L_{x,T}^p$ to denote $L_t^p(0,T;L_x^p)$.
	
	Given $H$ a Hilbert space, we denote by $L_2^0(L^2(\mathbb{R});H)$ the space of Hilbert-Schmidt operators from $L^2(\mathbb{R})$ into $H$. Its norm is given by
	$$ \left\|\Phi^{(k)}\right\|^2_{L_2^0(L^2(\mathbb{R}),H)}=\sum_{i\in\mathbb{N}} \left\| \Phi^{(k)} e_i \right\|^2_H.$$
	When $H=H^s(\mathbb{R})$, we write ${L_2^0(L^2(\mathbb{R}),H)}=L_2^{0,s}$.
	
	For brevity, let $S(t)=e^{it\partial_{x}^2},\quad U(t)=e^{-t\partial_{x}^3}$.
	For any $T>0$, the mild solution of \eqref{smodel} is
	\begin{equation}
		\begin{aligned}
			w_t=&\left(S(t)u_0,U(t)v_0\right) + \int_0^t(S(t-s)\Phi^{(1)} dW^{(1)}_s,U(t-s)\Phi^{(2)} dW_s^{(2)}) \\
			&+\int_0^t(S(t-s)N_1(w_s),U(t-s)N_2(w_s)) ds, \ \ t \in [0,T].
		\end{aligned}
	\end{equation}
	
	In this paper, we firstly concern the well-posedness of linear stochastic S-KdV:
	\begin{equation}\label{lsskdv}
		\left\{\begin{aligned}
			&d\breve{w}=L(\breve{w})dt+(\Phi^{(1)} dW_t^{(1)},\Phi^{(2)}dW^{(2)}_t), \\
			&\breve{w}|_{t=0}=(0,0).
		\end{aligned}\right.
	\end{equation}
	The mild solution of \eqref{lsskdv} is
	\begin{equation}\label{solution_of_lsskdv}
		\breve{w}_t=(\breve{u}_t,\breve{v}_t)=\int_0^t(S(t-s)\Phi^{(1)} dW^{(1)}_s,U(t-s)\Phi^{(2)} dW_s^{(2)}), \ t\in [0,T].
	\end{equation}
	
	By virtue of the Christ-Kiselev lemma in \cite{molinet2004well} and the Leibniz-type estimates in \cite{benea2016multiple}, we  use the below workspace to prove the well-posedness of \eqref{smodel}:
	\begin{equation*}
		\begin{array}{cl}
			X_{\sigma}(T)=&\{ u \in C([0,T]; H^\sigma(\mathbb{R})),\ \  u\in  L^2(\mathbb{R};L^{\infty}([0,T])),\\
			& \ v \in C([0,T];H^\sigma(\mathbb{R})),\ \  v \in L^2(\mathbb{R}; L^{\infty}([0,T])) \}.
		\end{array}
	\end{equation*}
	This space is much conciser than the workspace in \cite{dkdv}.
	We also denote
	$$\|(u,v)\|_{X_{\sigma}(T)}=\|u\|_{X_{\sigma}^1(T)}+\|v\|_{X_{\sigma}^2(T)}.$$ 
	
	As an example, for the case $\sigma=2$, we briefly show why this space works. We need the following lemma.
	
	\begin{lemma}[Theorem 4 in \cite{benea2016multiple}]\label{lem:leibnitz}
		For $s\geq 0$, $1<q_1,r_1,q_2,r_2\leq \infty$, $1/q_1+1/q_2 = 1/q$, $1/r_1+1/r_2 = 1/r$, $1\leq q,r<\infty$, we have
		\begin{align*}
			\|J^s(uv)\|_{L_x^rL_t^q}&\lesssim \|J^su\|_{L_x^{r_1}L_t^{q_1}}\|v\|_{L_x^{r_2}L_t^{q_2}}+\|J^sv\|_{L_x^{r_1}L_t^{q_1}}\|u\|_{L_x^{r_2}L_t^{q_2}}.
		\end{align*}
	\end{lemma}
	\begin{rem}
		Theorem 4 in \cite{benea2016multiple} is: $\forall~\alpha,\beta>0$,
		\begin{align*}
			&\quad\|D_1^\alpha D_2^\beta (fg)\|_{L_x^{s_1}L_y^{s_2}}\\
			&\lesssim \|D_1^{\alpha}D_2^\beta f\|_{L_x^{p_1}L_y^{p_2}}\|g\|_{L_x^{q_1}L_y^{q_2}}+\| f\|_{L_x^{p_3}L_y^{p_4}}\|D_1^{\alpha}D_2^\beta g\|_{L_x^{q_3}L_y^{q_4}}\\
			&\quad+\|D_1^{\alpha}f\|_{L_x^{p_5}L_y^{p_6}}\|D_2^\beta g\|_{L_x^{q_5}L_y^{q_6}}+\|D_2^\beta f\|_{L_x^{p_7}L_y^{p_8}}\|D_1^{\alpha}g\|_{L_x^{q_7}L_y^{q_8}}
		\end{align*}
		whenever $1<p_j,q_j\leq \infty$, $\max\{1/2,1/(1+\alpha)\}<s_1<\infty$, $1\leq s_2<\infty$, and the indices satisfy the natural H\"{o}lder-type conditions. In fact, it is easy to see that the inequality also holds for $\alpha = 0$ or $\beta = 0$ by the proof. See also \cite{benea2017quasi} and Theorem 1 in \cite{benea2020mixed}. In \cite{kenig1993well}, they cannot manipulate the endpoint case $r = 1$, $q_2 = \infty$. Thus, they use Strichartz estimate and Kato smoothing estimate to avoid this. 
		
		For us, using such strong estimate, the proof of local well-posedness is easier. Note that Lemma \ref{lem:leibnitz} is the vector-valued version of classical fractional Leibniz estimates. See for example \cite{coifman48wavelets}, \cite{muscalushlag}. 
	\end{rem}
	
	For \eqref{model} in $t\in [0,T]$, by  the algebra property of $H^1$, the smoothing effect of $U(t)$ and Lemma \ref{lem:leibnitz}, we have
	\begin{align*}
		\|u(t)\|_{L_t^\infty H^2_x}
		&\leq C(T)\|u_0\|_{H_x^2}+TC(\| u\|_{L_t^\infty H^2_x}\| v\|_{L_t^\infty H^1_x}+\| u\|_{L_t^\infty H^1_x}\| v\|_{L_t^\infty H^2_x}\\
		&\quad\hspace{100pt} +\|u\|^2_{L_t^\infty H^1_x}\|u\|_{L_t^\infty H^2_x}),\\
		\|v(t)\|_{L_t^\infty H^2_x}&\leq C(T)\|v_0\|_{H_x^2}+T^{\frac{1}{2}}C(T)(\|u\|_{L^\infty_tH_x^2}\|u\|_{L_x^2L_t^\infty}\\
		&\quad\hspace{100pt} +\|v\|_{L^\infty_tH_x^2}\|v\|_{L_x^2L_t^\infty}).
	\end{align*}
	By Christ-Kiselev lemma, if we want to estimate
	$$
	\left\|\int_{0}^{t} U(t-s)\partial_{x}(\gamma_2|u|^2-\frac{1}{2}v^2) ds\right\|_{L_x^2L_{t\in[0,1]}^\infty},
	$$
	we only need to estimate 
	$$
	\left\|\int_{0}^{1} U(t-s)\partial_{x}(\gamma_2|u|^2-\frac{1}{2}v^2) ds\right\|_{L_x^2L_{t\in[0,1]}^\infty}.
	$$
	Thus, by the smoothing effect of $U(t)$ and Lemma \ref{lem:leibnitz}, we have
	\begin{align*}
		\| u(t)\|_{L_x^2L_t^\infty}&\leq \tilde{C}(T)\|u_0\|_{H_x^{1}}+T\tilde{C}(T)(\| u\|_{L_t^\infty H^{1}_x}\| v\|_{L_t^\infty H^{1}_x}\\
		&\quad \ \ \ \ \ \  \ \ \ \ \ \ \ \ \ \ \ \ \ \ \ \ \ \ \ \ \ \ \ \ 
		+\|u\|_{L_t^\infty H^{1}_x}\|u\|^2_{L_t^\infty H^{1}_x}),\\
		\| v(t)\|_{L_x^2L_t^\infty}&\leq \tilde{C}(T)\|v_0\|_{H_x^{1}}+(1+T)^{\frac{1}{2}}T^{\frac{1}{2}}C(\|u\|_{L^\infty_tH_x^{1}}\|u\|_{L_x^2L_t^\infty}\\
		&\quad \ \ \ \ \ \  \ \ \ \ \ \ \ \ \ \ \ \ \ \ \ \ \ \ \ \ \ \ \ \ +\|v\|_{L^\infty_tH_x^{1}}\|v\|_{L_x^2L_t^\infty}).
	\end{align*}
	Here $C(T),\tilde{C}(T)$ are non-decrease with respect to $T$. 
	
	According to the above illustration and Leibniz-type estimates, it is easy to prove the following lemma.
	\begin{lemma}\label{lem:workspace}
		For any $ \sigma \in\mathbb{N^+} $ and $ T>0 $, we have 
		
		$i).$ $$\|(S(t)u_0,U(t)v_0)\|_{X_{\sigma}(T)}\leq C(T,\sigma)(\|u_0\|_{H^\sigma_x}+\|v_0\|_{H^\sigma_x});$$
		
		$ii).$
		\begin{equation*}
			\begin{aligned}
				&\quad\left\|\int_0^tS(t-s)N_1(w_s) ds\right\|_{X_{\sigma}^{1}(T)} \\
				&\leq  C(T,\sigma ,\gamma_1,\beta) T \left(\|u\|_{X_{\sigma}^{1}(T)}(\|u\|^2_{X_{1}^{1}(T)}+\|v\|_{X_{1}^{1}(T)})+\|v\|_{X_{\sigma}^{2}(T)}\|u\|_{X_{1}^{1}(T)}\right);
			\end{aligned}	
		\end{equation*}
		
		$iii).$
		\begin{equation*}
			\begin{aligned}
				&\quad\left\|\int_0^tU(t-s)N_2(w_s) ds\right\|_{X_{\sigma}^{2}(T)}\\
				&\leq C(T,\sigma,\gamma_2) T^{\frac{1}{2}}\left( \| v\|_{X_{\sigma}^{2}(T)}\| v\|_{X_{1}^{2}(T)} +\| u\|_{X_{\sigma}^{1}(T)}\| u\|_{X_{1}^{1}(T)} \right).
			\end{aligned}
		\end{equation*}
		Here $C(T,\sigma)$, $C(T,\sigma ,\gamma_1,\beta)$ and $ C(T,\sigma,\gamma_2)$ are non-decreasing with respect to $T$.
	\end{lemma}
	Hence, we can see the choice of workspace is appropriate.
	
	Now, we turn to properties of the solution of \eqref{lsskdv}. In this paper, the following proposition which is similar to the Theorem 3.2 in \cite{KdVH1}, will be needed. 
	However, because of the choice of workspace (our workspace doesn't need $\|\partial_x \cdot\|_{L_x^\infty L_t^2}$), the spatial regularity of $\breve{w}$ can be equal to the spatial regularity of $\Phi$.
	
	\begin{prop}\label{prop:linear}
		Suppose $\sigma \in \mathbb{N^+}$. Then,
		\begin{equation}\label{L2esti}
			\mathbb{E}\left( \|\breve{w}\|_{X_{\sigma}(T)}^2 \right) \leq C(T,\sigma)\left(\|\Phi^{(1)}\|_{L_2^{0,\sigma}}^2+\|\Phi^{(2)}\|_{L_2^{0,\sigma}}^2\right)
		\end{equation}
		for any $T>0$.
	\end{prop}
	
	According to \cite{KdVH1}, we only need to prove 
	$$ \mathbb{E}\left( \|\breve{u}\|_{L_T^{\infty}H^\sigma_x}^2\right) \leq TC\|\Phi^{(1)}\|_{L_2^{0,\sigma}}^2\ \ \text{and}  \ \ \mathbb{E}\left( \|\breve{u}\|_{L^2_xL^{\infty}_T}^2 \right) \leq C(T,\sigma)\|\Phi^{(1)}\|_{L_2^{0,\sigma}}^2.$$
	Here, $C(T,\sigma)$ is non-decreasing with respect to $T$.
	\begin{lemma}\label{lem:Linfty_TH^s_x}
		We have
		\begin{equation}\label{Linfty_TH^s_x}
			\mathbb{E}\left(\|\breve{u}\|_{L_T^{\infty}H^\sigma_x}^2 \right) \leq TC\|\Phi^{(1)}\|_{L_2^{0,\sigma}}^2,\ \
		\end{equation}
	\end{lemma}
	\begin{proof}
		By the unitary property of $S(t)$ and It\^{o} formula, we have
		\begin{equation*}
			\begin{array}{l}
				\|\hat{u}\|_{H^\sigma_x}^2
				=(J^\sigma\int_0^tS(-\tau)\Phi^{(1)} dW_{\tau}^{(1)},J^\sigma\int_0^tS(-\tau)\Phi^{(1)} dW_{\tau}^{(1)})\\
				\hspace{34pt}=2\int_0^t( J^\sigma\int_0^{\tau}S(-\iota)\Phi^{(1)} dW_{\iota}^{(1)}, J^\sigma S(-\tau)\Phi^{(1)} dW_{\tau}^{(1)} )\\
				\hspace{34pt}\quad+\sum_{i=0}^{\infty}\int_0^t (J^\sigma S(-\tau)\Phi^{(1)} e_i, J^\sigma S(-\tau)\Phi^{(1)} e_i) d\tau\\
				\hspace{34pt}=2\int_0^t( J^\sigma \int_0^{\tau}S(-\iota)\Phi^{(1)} dW_{\iota}^{(1)}, J^\sigma S(-\tau)\Phi^{(1)} dW_{\tau}^{(1)} )+t\|\Phi^{(1)}\|^2_{L^{0,\sigma}_2}.
			\end{array}
		\end{equation*}
		By the BDG inequality and Young inequality, we have
		\begin{equation*}
			\begin{array}{l}
				\quad\mathbb{E} \left(\sup_{t\in[0,T]}  \|\breve{u}\|_{H^\sigma_x}^2 \right)\\
				\leq  \mathbb{E}\left(C\sum_{i=0}^{\infty}\int_{0}^{T} ( J^\sigma\int_0^{\tau}S(-\iota)\Phi^{(1)} dW_{\iota}^{(1)}, J^\sigma S(-\tau)\Phi^{(1)}e_i )^2 d\tau\right)^{\frac{1}{2}}\\
				\quad+ T\|\Phi^{(1)}\|^2_{L^{0,\sigma}_2}\\
				=\mathbb{E}\left(C\sum_{i=0}^{\infty}\int_{0}^{T} ( J^\sigma\breve{u}(\tau), J^\sigma\Phi^{(1)}e_i )^2 d\tau\right)^{\frac{1}{2}} + T\|\Phi^{(1)}\|^2_{L^{0,\sigma}_2}\\
				\leq \frac{1}{2}\mathbb{E} \left(\sup_{t\in[0,T]}  \|\breve{u}\|_{H^\sigma_x}^2 \right)+ TC\|\Phi^{(1)}\|^2_{L^{0,\sigma}_2}	
			\end{array}
		\end{equation*}
		Thus, we finish the proof.
	\end{proof}
	
	\begin{rem}
		Using the arguments as Theorem 6.4 in \cite{dp}, it can be proved that $\breve{u} \in L^2(\Omega; C([0,T];H^\sigma(\mathbb{R})))$.
	\end{rem}	
	
	\begin{lemma}\label{lem:L_x^2L_T^infty}
		We have
		\begin{equation}\label{L_x^2L_T^infty}
			\mathbb{E}\left( \|\breve{u}\|_{L^2_xL^{\infty}_T}^2  \right) \leq C(T,s)\|\Phi^{(1)}\|_{L_2^{0,s}}^2
		\end{equation}
		for any $s>{1}/{2}$.
	\end{lemma}
	Although the proof of Lemma \ref{lem:L_x^2L_T^infty} is complex, it is similar to the Proposition 3.2 in \cite{KdVH1}. We omit it. 
	
	\section{Priori Estimate}\label{chapterpriori}
	In this section, for \eqref{smodel}, we firstly show the local well-posedness by a pathwise fixed point argument in $X_{1}(T(\omega))$. Then, we construct approximation equations of the \eqref{smodel}. Finally, supposing the solutions of the approximation equations are smooth enough, we get a uniform estimate of these equations.

	To prove the local well-posedness of \eqref{smodel} in $X_{1}(T(\omega))$, we introduce the following notations.
	
	For any $d,T>0$ and $\sigma\in \mathbb{N^+}$, let
	$$ B_d^\sigma(T)=\{{(u,v)\in X_\sigma(T): \|(u,v)\|_{X_\sigma(T)}}\leq d\}.$$
	For any $(u_0,v_0)$ satisfying the conditions of Theorem \ref{thm:main}, the random mapping $\Psi (\cdot,\cdot)(\omega)$ from $(\mu,\nu)\in X_\sigma(T(\omega))$ to $ \Psi (\mu,\nu) \in X_\sigma(T(\omega)) $ is defined as follow:
	\begin{equation*}
		\begin{aligned}
			&\quad\Psi ((\mu,\nu))(\omega) \\
			&=\left(\Psi^1 ((\mu,\nu)),\Psi^2 ((\mu,\nu))\right)\\
			&=(S(t)u_0, U(t)v_0) +\int_0^t(S(t-s)N_1((\mu,\nu)),U(t-s)N_2((\mu,\nu))) ds + \breve{w}(\omega).
		\end{aligned}
	\end{equation*}
	
	According to Proposition \ref{prop:linear} and Lemma \ref{lem:workspace}, it is clear that for any $d,T>0$, we have
	\begin{align*}
		\|\Psi(u,v)(\omega) \|_{X_{1}(T)}
		&\leq C(T)(\|u_0(\omega)\|_{H^{1}_x}+\|v_0(\omega)\|_{H^{1}_x}) + C(T)\|\breve{w}(\omega)\|_{X_{1}(T)}\\
		&+ C(T,\gamma_1,\gamma_2,\beta) T^{\frac{1}{2}}\left(\|(u,v)\|^2_{X_{1}(T)}+\|(u,v)\|^3_{X_{1}(T)}\right)
		\ a.s.\mathbb{P}
	\end{align*}
	and 
	\begin{equation*}
		\begin{aligned}
			&\quad\|\Psi(u,v)(\omega) -\Psi(\tilde{u},\tilde{v})(\omega) \|_{X_{1}(T)} \\
			&\leq   \tilde{C}(T,\gamma_1,\gamma_2,\beta) T^{\frac{1}{2}} (\|(u-\tilde{u},v-\tilde{v})\|_{X_{1}(T)} (\|(u,v)\|_{X_{1}(T)}+\|(\tilde{u},\tilde{v})\|_{X_{1}(T)})\\
			& \ \ \ \ \ \ \ \ \ \ \ \  +\|(u-\tilde{u},v-\tilde{v})\|_{X_{1}(T)}\cdot (\|(u,v)\|^2_{X_{1}(T)}+\|(\tilde{u},\tilde{v})\|^2_{X_{1}(T)})) \ a.s.\mathbb{P}.
		\end{aligned}
	\end{equation*}
	Here, $C(T)$, $ C(T,\gamma_1,\gamma_2,\beta)$, $\tilde{C}(T,\gamma_1,\gamma_2,\beta)$ are non-decreasing with respect to $T$. Thus, as long as we choose
	$$
	d=2C(T)(\|u_0(\omega)\|_{H^{1}_x}+\|v_0(\omega)\|_{H^{1}_x}+\|\breve{w}(\omega)\|_{X_{1}(T)})
	$$
	and
	$$T(\omega)= \left(8C(T,\gamma_1,\gamma_2,\beta) (d+d^2)\right)^{-2} \wedge \left(8\tilde{C}(T,\gamma_1,\gamma_2,\beta)(d^2+d)\right)^{-2},$$
	which is strictly larger than $0$,
	$\Psi (\cdot,\cdot)(\omega)$ is bounded and contracting on $ B_{d}^{1}(T(\omega)) $ in $[0, T(\omega)]$ almost surely.
	
	Therefore, $\Psi(\cdot,\cdot)(\omega)$ has a unique fixed point in $ B_{d}^{1}(T(\omega)) $  almost surely. Here, the existence interval of the solution relies on $\omega$. 
	
	To extend the solution to the common interval $[0, T]$, we need a $H_x^1$ priori estimate of $(u,v)$. However, because in the  proof of the high regular well-posedness of the forced system will appear products of forces and $u,v,\partial_x u,\partial_xv$, it is difficult to get the priori estimate for $(u,v)$ pathwisely. Thus, we need to construct a sequence of approximation equations.
	
	We do the following smoothing treatments of noises and initial values.
	Let
	$P_m = \mathscr{F}^{-1}\chi_{[-m,m]}(\xi)\mathscr{F}$, 
	and
	$$\{\Phi_m\}_{m\in\mathbb{N}^+} = \{(\Phi^{(1)}_m , \Phi^{(2)}_m) \}_{m\in\mathbb{N}^+}=\{(P_m\Phi^{(1)},P_m\Phi^{(2)})\}_{m\in\mathbb{N^+}} \subset L^{0,2}_2\times L^{0,3}_2.$$
	It is clear that
	$$ (\Phi^{(1)}_m, \Phi^{(2)}_m) \longrightarrow (\Phi^{(1)} , \Phi^{(2)}) \ \ \text{in} \  \ L_2^{0,1} \times L_2^{0,1},\quad m\rightarrow \infty.$$ 	
	Let $\{(u_{m}(0),v_{m}(0))\}_{m\in\mathbb{N}^+}=\{(P_mu_{0},P_mv_{0})\}_{n\in\mathbb{N}^+} $ be a sequence in $H^2(\mathbb{R})\times H^3(\mathbb{R}) $ almost surely. By the Sobolev embedding, the Carleson-Hunt theorem in \cite{mordenFA} and the dominated convergence theorem, one has
	$$u_{m}(0)\longrightarrow u_0 \ \ \text{in} \ \  L^2(\Omega; H^{1}(\mathbb{R})) \cap L^4(\Omega;L^4(\mathbb{R}))\cap L^{10}(\Omega;L^2(\mathbb{R})),
	$$
	$$
	v_{m}(0)\longrightarrow v_0  \  \ \text{in}  \  \  L^2(\Omega; H^{1}(\mathbb{R})) \cap L^{3}(\Omega;L^3(\mathbb{R})). $$

We first introduce the approximation equations with high regularity noises and initial datum:
\begin{equation}\label{smoothinitialandnoise}
	\left\{
		\begin{array}{rl}
			\begin{aligned}
				&d u_m=i\partial_{xx}u_mdt -i(\gamma_1 u_mv_m+\beta |u_m|^2u_m)dt + \Phi_m^{(1)} dW_t^{(1)},\\
				&d v_m=-\partial_{xxx}v_mdt+  \partial_x(\gamma_2|u_m|^2-v_m^2/2)dt+ \Phi_m^{(2)}dW_t^{(2)},\\
				&u_m(0)=P_mu_0,\  v_m(0)=P_mv_0.
			\end{aligned}
		\end{array}
	\right.
\end{equation}
	
	Then, we propose the following approximation equations of \eqref{smoothinitialandnoise}, whose paths will be proved almost surely in $C([0,T];H^2_x\times H^3_x)$ and have enough conservation laws at the same time. 
	\begin{equation}\label{approxila}
		\left\{
		\begin{array}{ll}
			d u_{m,n,K}=&i\partial_{xx}u_{m,n,K}dt -i\gamma_1 \psi_K(|u_{m,n,K}|^2)u_{m,n,K}v_{m,n,K}dt\\
			&-i\beta \varphi_K(|u_{m,n,K}|^2)|u_{m,n,K}|^2u_{m,n,K}dt + \Phi_m^{(1)} dW_t^{(1)},\\
			d v_{m,n,K}=&-\partial_{xxx}v_{m,n,K}dt+  P_n\partial_x(\gamma_2\varphi_K(|u_{m,n,K}|^2)|u_{m,n,K}|^2)dt\\
			&-\frac{1}{2}P_n\partial_{x}(\varphi_K(v_{m,n,K})v_{m,n,K}^2)dt+ \Phi_m^{(2)}dW_t^{(2)},\\
			u_{m,n,K}(0)=&u_m(0),\  v_{m,n,k}(0)=v_m(0).
		\end{array}
		\right.
	\end{equation}
	In the whole paper, $n,K\in \mathbb{N^+}$, $n\geq m$. $\varphi\in C_0^\infty$ is a real cut-off function satisfying  $\varphi|_{[-1,1]} = 1$. $\varphi_K(x) = \varphi(x/K)$, $\psi_K(x) = x\varphi_K'(x)+\varphi_K(x)$. In Proposition \ref{prop:priori_esti}, we will use the fact that $v_{m,n,K}=P_nv_{m,n,K}$.
	
	\begin{prop}\label{prop:priori_esti}
		For any $T>0$, suppose that $\{(u_{m,n,K},v_{m,n,K})\}_{m,n\in\mathbb{N^+}}$ are in $C([0,T];H^2_x\times H^3_x)$ almost surely. Then, under conditions of Theorem \ref{thm:main},  $\{(u_{m,n,K},v_{m,n,K})\}_{m,n\in\mathbb{N^+}}$ are bounded in $L^2(\Omega;L^\infty(0,T;\mathcal{H}^1(\mathbb{R})))$ and have the estimate \eqref{priori_esti_finish}, which does not depend on $m,n,K$.
	\end{prop}
	\begin{proof}
		For any $k\in\mathbb{N^+}, t\geq0$, we apply It\^{o} formula to 
		\begin{equation}\label{u^k_n_L_x^2}
			\begin{aligned}
				&\quad\|u_{m,n,K}\|^{2k}_{L_x^2}\\
				&=\|u_{m}(0)\|^{2k}_{L_x^2}+k\int_{0}^{t} \|u_{m,n,K}\|_{L_x^2}^{2k-2}\left(u_{m,n,K},\Phi_m^{(1)}dW_s^{(1)}\right)\\
				&\quad+k \int_{0}^{t} \|u_{m,n,K}\|_{L_x^2}^{2k-2} \left(\bar{u}_{m,n,K},\Phi_{m}^{(1)}dW_s^{(1)}\right)\\
				&\quad+k\|\Phi_{m}^{(1)}\|^2_{L_2^{0,0}}\int_{0}^{t} \|u_{m,n,K}\|_{L_x^2}^{2(k-1)} ds\\
				&\quad+k(k-1)\int_{0}^{t} \|u_{m,n,K}\|_{L_x^2}^{2(k-2)}\Bigg(\left(u_{m,n,K},\Phi_{m}^{(1)}e_i\right)\left(\bar{u}_{m,n,K},\Phi_m^{(1)}e_i\right)\\
				&\qquad\qquad\hspace{40pt}+\frac{1}{2}\left(\left(u_{m,n,K},\Phi_m^{(1)}e_i\right)^2+ \left(\bar{u}_{m,n,K},\Phi_m^{(1)}e_i\right)^2\right)\Bigg) ds.
			\end{aligned}
		\end{equation}
		It follows by BDG's inequality that
		\begin{equation}\label{max_u^k_n_L_x^2}
			\mathbb{E} \Big(\sup_{t\in[0,T]}|u_{m,n,K}|^{2k}_{L_x^2}\Big)\leq C(k)\mathbb{E}(|u_m(0)|^{2k}_{L_x^2})+C(k)T^{k} \|\Phi_m^{(1)}\|^{2k}_{L_2^{0,0}},
		\end{equation}
		for any $T>0$.
		
		Let
		\begin{equation*}
			\mathcal{I}_t(u_{m,n,K},v_{m,n,K}):=\int_{\mathbb{R}} \text{Im}(u_{m,n,K}\partial_x\bar{u}_{m,n,K}) + \frac{\gamma_1}{2\gamma_2}v_{m,n,K}^2 dx.
		\end{equation*}
		Similarly, by It\^{o} formula, one has
		\begin{equation}\label{I(u,v)}
			\begin{aligned}
				&\quad\mathcal{I}_t(u_{m,n,K},v_{m,n,K})\\ &=i\int_{0}^{t}\left(\partial_x(u_{m,n,K}-\bar{u}_{m,n,K}),\Phi^{(1)}_mdW_s^{(1)}\right)+t\sum_{i=0}^{\infty}\left(\Phi_m^{(1)} e_i,\partial_{x} \Phi_m^{(1)}e_i\right)\\
				&\quad+\frac{\gamma_1}{2\gamma_2}t\|\Phi_m^{(2)}\|^2_{L_2^{0,0}}+\frac{\gamma_1}{\gamma_2}\int_{0}^{t} \left(v_{m,n,K},\Phi_m^{(2)}dW_s^{(2)}\right).
			\end{aligned}
		\end{equation}
		What's more, we consider 
		\begin{equation*}
			\begin{aligned}
				&\quad\mathcal{E}_t(u_{m,n,K},v_{m,n,K})\\
				&=\int_{\mathbb{R}}|\partial_x u_{m,n,K}|^2+\frac{\gamma_1}{2\gamma_2}(|\partial_x v_{m,n,K}|^2-\psi_{2,K}(v_{m,n,K}))\\
				&\quad+\gamma_1\varphi_K(|u_{m,n,K}|^2)|u_{m,n,K}|^2v_{m,n,K}
				+\beta \psi_{1,K}(|u_{m,n,K}|^2)~dx,
			\end{aligned}
		\end{equation*}
		where $$\psi_{1,K}(x) =  \int_0^x s\varphi_K(s)~ds,\quad\psi_{2,K}(x) = \int_0^x s^2\varphi_K(s)~ds.$$
		By It\^{o} formula and the cut-off property of $\varphi_K(\cdot)$, it follows
		\begin{equation}\label{E(u,v)}
			\begin{array}{l}
				\quad\mathcal{E}_t(u_{m,n,K},v_{m,n,K})
				\\
				=\int_{\mathbb{R}} \left| \frac{\partial u_m(0)}{\partial x} \right| ^2 dx+\frac{\gamma_1}{2\gamma_2}	\int_{\mathbb{R}} \left(\left|\frac{\partial v_m(0)}{\partial x} \right| ^2-\psi_{2,K}(v_m(0))\right) dx
				\\
				+\gamma_1\int_{\mathbb{R}} \varphi_K(|u_m(0)|^2)|u_m(0)|^2v_m(0) dx+\beta\int_{\mathbb{R}} \psi_{1,K}(|u_m(0)|^2) dx
				\\
				+t\|\Phi_m^{(1)}\|^2_{L_2^{0,1}} +\int_{0}^{t} \left( \frac{\partial u_{m,n,K}}{\partial x}+\frac{\partial \bar{u}_{m,n,K}}{\partial x}, \frac{\partial}{\partial x} \Phi^{(1)}_m dW^{(1)}_s \right)+\frac{t\gamma_1}{2\gamma_2} \|\Phi_m^{(2)}\|^2_{L_2^{0,1}}
				\\
				+\frac{\gamma_1}{\gamma_2}\int_{0}^{t} \left(\frac{\partial v_{m,n,K}}{\partial x},\frac{\partial  }{\partial x}\Phi_m^{(2)}dW_s^{(2)}\right)
				\\
				+\beta\int_{0}^{t}\left(|u_{m,n,K}|^2u_{m,n,K}+|u_{m,n,K}|^2\bar{u}_{m,n,K},\varphi_K(|u_{m,n,K}|^2)\Phi_m^{(1)}dW_s^{(1)}\right)
				 \\
				+\frac{\beta}{2}\sum_{i=0}^{\infty}\int_{0}^{t} \left(\left(u_{m,n,K}+\bar{u}_{m,n,K}\right)^2\varphi_K(|u_{m,n,K}|^2)\Phi_m^{(1)}e_i,\Phi_m^{(1)}e_i\right) ds
				\\
				+\frac{\beta}{2}\sum_{i=0}^{\infty}\int_{0}^{t} \left(\left(u_{m,n,K}+\bar{u}_{m,n,K}\right)^2|u_{m,n,K}|^2\varphi'_K(|u_{m,n,K}|^2)\Phi_{m}^{(1)}e_i,\Phi_m^{(1)}e_i\right) ds
				\\
				+\beta\sum_{i=0}^{\infty}\int_{0}^{t} \left(|u_{m,n,K}|^2\varphi_K(|u_{m,n,K}|^2)\Phi_m^{(1)}e_i,\Phi_m^{(1)}e_i\right) ds
				\\
				+\gamma_1\int_{0}^{t}\left(\varphi_K(|u_{m,n,K}|^2)v_{m,n,K}(\bar{u}_{m,n,K}+u_{m,n,K}),\Phi_m^{(1)}dW_s^{(1)}\right)
				\\
				+\gamma_1\int_{0}^{t}\left(\varphi'_K(|u_{m,n,K}|^2)|u_{m,n,K}|^2v_{m,n,K}(\bar{u}_{m,n,K}+u_{m,n,K}),\Phi_m^{(1)}dW_s^{(1)}\right)
				\\
				+\gamma_1\int_{0}^{t} \left(\varphi_K(|u_{m,n,K}|^2)|u_{m,n,K}|^2,\Phi_m^{(2)}dW_s^{(2)}\right)
				\\
				+\gamma_1\sum_{i=0}^{\infty}\int_0^t  \left(v_{m,n,K}\varphi_K(|u_{m,n,K}|^2)\Phi_{m}^{(1)}e_i,\Phi_m^{(1)}e_i\right) ds
				\\
				+\frac{\gamma_1}{2}\sum_{i=0}^{\infty}\int_0^t  \left(v_{m,n,K}\varphi'_K(|u_{m,n,K}|^2)\Phi_m^{(1)}e_i(u_{m,n,K}+\bar{u}_{m,n,K})^2,\Phi_m^{(1)}e_i\right) ds
				\\
				+\gamma_1\sum_{i=0}^{\infty}\int_0^t  \left(v_{m,n,K}|u_{m,n,K}|^2\varphi'_K(|u_{m,n,K}|^2)\Phi_{m}^{(1)}e_i,\Phi_{m}^{(1)}e_i\right) ds
				\\				+2\gamma_1\sum_{i=0}^{\infty}\int_0^t  \Big(v_{m,n,K}|u_{m,n,K}|^2\varphi''_K(|u_{m,n,K}|^2)\Phi_{m}^{(1)}e_i\mathrm{Re}(u_{m,n,K})^2,\Phi_m^{(1)}e_i\Big) ds
				\\				-\frac{\gamma_1}{2\gamma_2}\int_{0}^{t} \left(\varphi_K(v_{m,n,K})v_{m,n,K}^2,\Phi^{(2)}_mdW^{(2)}_s\right)
				\\				-\frac{\gamma_1}{4\gamma_2} \sum_{i=0}^{\infty} \int_{0}^{t}\left( 2v_{m,n,K}\varphi_K(v_{m,n,K})+v_{m,n,K}^2\varphi'_K(v_{m,n,K}) , (\Phi_m^{(2)}e_i)^2 \right) ds.
			\end{array}
		\end{equation}
		Now we would use $ \|u_{m,n,K}\|^{2k}_{L_x^2},\  \mathcal{I}_t(u_{m,n,K},v_{m,n,K}),\  \mathcal{E}_t(u_{m,n,K},v_{m,n,K})$ to construct an upper bound of $\|u_{m,n,K}\|^2_{H_x^1}+\|v_{m,n,K}\|^2_{H_x^1}$.
		
		Note that 
		$$ \int_{\mathbb{R}} \psi_{1,K}(|u_{m,n,K}|^2) ~dx\leq \frac{1}{2}\|u_{m,n,K}\|^4_{L_x^4 },~~ \int_{\mathbb{R}} \psi_{2,K}(v_{m,n,K}) ~dx\leq \frac{1}{3}\|v_{m,n,K}\|^3_{L_x^3}.$$
		According to Littlewood-Paley square function theorem, Minkowski inequality and Bernstein inequality, we have  $\|v_{m,n,K}\|^3_{L_x^3} \leq C\|v_{m,n,K}\|^3_{\dot{H}_x^{{1}/{6}}}$. Then by H\"{o}lder inequality, the definition of $\mathcal{I}_t$, Solobev inequality and Young inequality, we have
		\begin{align*}
			&\quad\|v_{m,n,K}\|_{L_x^3}^3\lesssim \|v_{m,n,K}\|_{\dot{H}_x^\frac{1}{6}}^3\lesssim \|v_{m,n,K}\|_{L_x^2}^\frac{5}{2}\|v_{m,n,K}\|_{\dot{H}_x^1}^\frac{1}{2}\\
			&\lesssim (I(u_{m,n,K},v_{m,n,K})^\frac{5}{4}+C\|u_{m,n,K}\|_{H_x^\frac{1}{2}}^\frac{5}{2})\|v_{m,n,K}\|_{\dot{H}_x^1}^\frac{1}{2},\\
			&\quad\int_{\mathbb{R}} \varphi_K(|u_{m,n,K}|^2)v_{m,n,K}|u_{m,n,K}|^2 dx\\
			&\leq \|u_{m,n,K}\|^2_{L_x^4}\|v_{m,n,K}\|_{L_x^2}\lesssim \|u_{m,n,K}\|^{\frac{3}{2}}_{L_x^2}\|u_{m,n,K}\|^{\frac{1}{2}}_{H_x^{1}} \|v_{m,n,K}\|_{L_x^2}\\
			&\leq C_\varepsilon \|u_{m,n,K}\|^{6}_{L_x^2}+ \varepsilon (\|u_{m,n,K}\|^{2}_{H_x^{1}}+ \|v_{m,n,K}\|^2_{L_x^2}) 
		\end{align*}
		and
		\begin{align*}
			\|u_{m,n,K}\|_{L_x^4}^4&\lesssim \|u_{m,n,K}\|_{H_x^\frac{1}{4}}^4\lesssim \|u_{m,n,K}\|_{L_x^2}^3\|u_{m,n,K}\|_{H_x^1}\\
			&\leq C_\varepsilon\|u_{m,n,K}\|_{L_x^2}^6 +\varepsilon \|u_{m,n,K}\|^2_{H_x^1}.
		\end{align*}
		Furthermore, by H\"{o}lder inequality, we have
		\begin{align*}
			\|u_{m,n,K}\|_{H_x^\frac{1}{2}}^\frac{5}{2}\|v_{m,n,K}\|_{\dot{H}_x^1}^\frac{1}{2}&\leq \|u_{m,n,K}\|_{L_x^2}^\frac{5}{4}\|u_{m,n,K}\|_{H_x^1}^\frac{5}{4}\|v_{m,n,K}\|_{\dot{H}_x^1}^\frac{1}{2}\\
			&\leq C_\varepsilon\|u_{m,n,K}\|_{L_x^2}^{10}+\varepsilon\|u_{m,n,K}\|_{H_x^1}^2+\varepsilon\|v_{m,n,K}\|_{\dot{H}_x^1}^2.
		\end{align*}
		Thus, for a fixed $T>0$, there exists $C_1, C_2>0$ such that
		\begin{equation}\label{u_nv_nH^1_esti}
			\begin{array}{rl}
				&\left(\frac{\gamma_1}{4\gamma_2} \wedge \frac{1}{2}\right)(\|u\|_{H_x^1}^2+\|v\|_{H_x^1}^2)\\
				\leq&C_1(\|u_{m,n,K}\|^{2}_{L_x^2}+\|u_{m,n,K}\|^{6}_{L_x^2}+	\|u_{m,n,K}\|^{10}_{L_x^2})+\mathcal{I}_t(u_{m,n,K},v_{m,n,K})\\
				&+C_2\mathcal{I}_t (u_{m,n,K},v_{m,n,K})^\frac{5}{3}+\mathcal{E}_t(u_{m,n,K},v_{m,n,K}),
			\end{array}
		\end{equation}
		for any $t \in [0, T]$.
		
		For the sake of simplicity, let $\mathcal{Q}_t(u_{m,n,K},v_{m,n,K})$ be the right side of \eqref{u_nv_nH^1_esti}. By \eqref{max_u^k_n_L_x^2}--\eqref{E(u,v)} and BDG's inequality, there exist positive constants from $C_3$ to $C_{14}$, which are not depend on $m,n,k$ such that
		\begin{equation}\label{sup_esti}
			\begin{array}{l}
				\quad\mathbb{E} \sup_{t\in[0,T]} {Q}_t(u_{m,n,K},v_{m,n,K})\\
				\leq
				\sum_{i=1,3,5}C_3 (\mathbb{E}(|u_{m}(0)|^{2i}_{L_x^2})+T^{i} \|\Phi_{m}^{(1)}\|^{2i}_{L_2^{0,0}})\\
				+C_3(\mathbb{E}(|u_m(0)|^{2}_{L_x^2})+T \|\Phi_m^{(1)}\|^{2}_{L_2^{0,0}}+T \|\Phi_m^{(1)}\|^{2}_{L_2^{0,1}})\\
				+\varepsilon\mathbb{E} \sup_{t\in[0,T]}|v_{m,n,K}|^2+C_4T\|\Phi_m^{(2)}\|^2_{L^{0,1}_2}+\frac{1}{2}T\|\Phi_m^{(2)}\|^2_{L_2^{0,0}}\\
				+C_5 ((T\|\Phi_m^{(2)}\|^2_{L_2^{0,1}})^{\frac{5}{3}}+\mathbb{E}|u_m(0)|^2_{L_x^2} +T\|\Phi_m^{(1)}\|^2_{L_2^{0,0}}+T^5\|\Phi^{(2)}_m\|^{10}_{L_2^{0,1}})\\
				+ \varepsilon\mathbb{E} \sup_{t\in[0,T]} |v_{m,n,K}|^2_{L_x^2}+C_6T^5\|\Phi^{(2)}_m\|^{10}_{L_2^{0,0}}\\
				+\mathbb{E}\left| {\partial_x u_m(0)} \right| ^2_{L_x^2} +\frac{\gamma_1}{2\gamma_2}	\mathbb{E}\left| {\partial_x v_m(0)} \right| ^2_{L_x^2} +\frac{\beta}{2}\mathbb{E} |u_m(0)|^4_{L_x^4}\\
				+\gamma_1\mathbb{E}\int_{\mathbb{R}} |v_m(0)||u_m(0)|^2  dx  \\
				+\frac{\gamma_1}{6\gamma_2}\mathbb{E}\| v_m(0)\|^3_{L_x^3} +T\|\Phi_m^{(1)}\|^2_{L_2^{0,1}} \\
				+\varepsilon\mathbb{E} \sup_{t\in[0,T]} |\partial_x u_{m,n,K}|^2_{L_x^2}+C_7T\|\Phi_m^{(1)}\|^2_{L_2^{0,1}}\\
				+\frac{\gamma_1}{2\gamma_2}T \|\Phi_m^{(2)}\|^2_{L_2^{0,1}}+\varepsilon\mathbb{E} \sup_{t\in[0,T]} |\partial_x v_{m,n,K}|^2_{L_x^2}+C_8T\|\Phi_m^{(2)}\|^2_{L_2^{0,1}}\\
				+C_9 T\|\Phi_m^{(1)}\|^2_{L_2^{0,1}}+\varepsilon\mathbb{E} \sup_{t\in[0,T]} |\partial_xu_{m,n,K}|^2_{L_x^2}\\
				+\varepsilon(\mathbb{E}(|u_m(0)|^{10}_{L_x^2})+T^{5} \|\Phi_m^{(1)}\|^{10}_{L_2^{0,0}})\\
				+ 3C_{10} T\|\Phi_m^{(1)}\|^{2}_{L_2^{0,1}}(\mathbb{E}(|u_m(0)|^{2}_{L_x^2})+T \|\Phi_{m,n,K}^{(1)}\|^{2}_{L_2^{0,0}})\\
				+3(C_{11}T^2\|\Phi_m^{(1)}\|^4_{L_2^{0,1}}+\varepsilon\mathbb{E} \sup_{t\in[0,T]} | v_{m,n,K}|^2_{L_x^2})\\
				+\varepsilon\mathbb{E} \sup_{t\in[0,T]} | v_{m,n,K}|^2_{L_x^2}+C_{12}T\|\Phi_m^{(1)}\|^2_{L_2^{0,0}}(\mathbb{E}(|u_m(0)|^{2}_{L_x^2})+T \|\Phi_m^{(1)}\|^{2}_{L_2^{0,0}})\\
				+C_{13}(\mathbb{E}(|u_m(0)|^{4}_{L_x^2})+T^{2} \|\Phi_m^{(1)}\|^{4}_{L_2^{0,0}}+T\|\Phi_m^{(2)}\|^2_{L_2^{0,1}})\\ +\varepsilon\mathbb{E} \sup_{t\in[0,T]} | v_{m,n,K}|^2_{L_x^2}+C_{14}T^{2} \|\Phi_m^{(2)}\|^{4}_{L_2^{0,1}} \\+\frac{\gamma_1}{2\gamma_2}T^{\frac{1}{2}}\|\Phi_m^{(2)}\|_{L_2^{0,1}} \mathbb{E} \sup_{t\in[0,T]}| v_{m,n,K}|^2_{L_x^2},
			\end{array}
		\end{equation}
		for any $T>0$.
		
		To get rid of the last term in \eqref{sup_esti}, we note that 
		\begin{equation}\label{v_n_L_infty}
			\begin{aligned}
				\mathbb{E} \sup_{t\in[0,T]}| v_{m,n,K}|^2_{L_x^2} & \leq  C_{15} (\mathbb{E}(|u_m(0)|^{2}_{L_x^2})+T \|\Phi_m^{(1)}\|^{2}_{L_2^{0,1}}+T\|\Phi_m^{(2)}\|^2_{L^{0,1}_2})\\
				&
				\quad+\varepsilon\mathbb{E} \sup_{t\in[0,T]}|v_{m,n,K}|^2_{L_x^2}+\varepsilon\mathbb{E} \sup_{t\in[0,T]}|\partial_x u_{m,n,K}|^2_{L_x^2}.
			\end{aligned}
		\end{equation}
		
		On the other hand, we have
		\begin{equation}\label{compare_Q}
				\mathbb{E} \sup_{t\in[0,T]} {Q}_t(u_{m,n,K},v_{m,n,K})\geq \left(\frac{\gamma_1}{4\gamma_2} \wedge \frac{1}{2}\right)\mathbb{E}\sup_{t\in[0,T]}(\|u\|_{H^1}^2+\|v\|_{H^1}^2).
		\end{equation}
		
		Hence, combining \eqref{sup_esti}, \eqref{v_n_L_infty} with \eqref{compare_Q}, we obtain
		\begin{equation}\label{priori_esti_finish}
			\begin{aligned}
				&\quad\mathbb{E}\sup_{t\in[0,T]}(\|u\|_{H^1}^2+\|v\|_{H^1}^2)\\
				\leq 
				& C(T,\gamma_1,\gamma_2,\beta)\Big( \sum_{i=1,5} \big(\mathbb{E}|u(0)|^{2i}_{L_x^2}+ \|\Phi^{(1)}\|^{2i}_{L_2^{0,1}}+\|\Phi^{(2)}\|^{2i}_{L_2^{0,1}}\big)+\mathbb{E}|v(0)|^{2}_{L_x^2} \\
				&\   +\mathbb{E}|v(0)|^{3}_{L_x^3} +\mathbb{E}\left| {\partial_x u(0)} \right| ^2_{L_x^2} +\mathbb{E}\left| {\partial_x v(0)} \right| ^2_{L_x^2}+\mathbb{E} |u(0)|^4_{L_x^4}\Big)				
			\end{aligned}
		\end{equation}
		which finishes the proof. 
	\end{proof}
	
	\section{The Global Well-Posedness of Approximation Equations}\label{chapterwellposedness}
	In this section, for any $T>0$, we prove $\{(u_{m,n,k}, v_{m,n,k})\}_{n\geq m,m,n,k\in \mathbb{N^+}}$ are well-posed in $ C([0,T];H^2_x\times H^3_x) $ almost surely. 
	We introduce the following linear approximation equations:
	\begin{equation}\label{linearapproxi}
		\left\{
		\begin{aligned}
			&d \breve{u}_{m,n}=i\partial_{xx}\breve{u}_{m,n}dt + \Phi_{m}^{(1)} dW_t^{(1)},\\
			&d \breve{v}_{m,n}=-\partial_{xxx}\breve{v}_{m,n}dt+ \Phi_m^{(2)}dW_t^{(2)},\\
			&\breve{u}_m(0)=0,\  \breve{v}_m(0)=0.
		\end{aligned}
		\right.
	\end{equation}
	Set $\tilde{u}_{m,n,K}=u_{m,n,K}-\breve{u}_{m,n}$. According to Proposition \ref{prop:linear}, $(\breve{u}_{m,n},\breve{v}_{m,n})\in C([0,T];H^2_x \times H^3_x)$. Thus, for the simplicity of notations, we only need to prove Proposition \ref{prop:globalwellposedness}, where $u,v,f,g$ actually represent
	$$
	u_{m,n,k}- \breve{u}_{m,n}, v_{m,n,k}-\breve{v}_{m,n},  \breve{u}_{m,n}, \breve{v}_{m,n},
	$$
	respectively. 
	\begin{prop}\label{prop:globalwellposedness}
		For any $ T>0, n\in \mathbb{N^+}$ and $f, g, u_0, v_0$ satisfying 
		$$
		(f,g)\in L_T^\infty H_x^2\times L_T^\infty H_x^3, \  \text{supp}~\hat{u}_0, \text{supp}~\hat{v}_0\in (-n,n),
		$$
		\begin{equation*}
			\left\{
			\begin{aligned}
				&i\partial_t u+\partial_{xx}u =(u+f)( \gamma_1 \psi_K(|u+f|^2)(v+g)+\beta \varphi_K(|u+f|^2)|u+f|^2),\\
				&\partial_t v+\partial_{xxx}v=  P_{n}\partial_x(\gamma_2\varphi_K(|u+f|^2)|u+f|^2-\varphi_K(v+g)(v+g)^2/2),\\
				&(u(0,x),v(0,x)=(u_0(x),v_0(x)),
			\end{aligned}
			\right.
		\end{equation*}
		is global well-posedness in $H^2(\mathbb{R})\times H_{n}(\mathbb{R})$, where 
		$$H_n=\{h\in L^2(\mathbb{R}):\mathrm{supp}~\mathscr{F}(h)\subset [-n,n]\}$$
		with the norm $\|h\|_{H_n}=\|h\|_{L^2}$.
	\end{prop}
	\begin{proof}
		Since $v$ is cut off, we first establish the local well-posedness in $L^2\times H_n$. Let 
		$$\|u\|_{\mathscr{X}_T}:=\|u\|_{C([0,T];L_x^2)}+\|u\|_{L_{x,T}^6},\quad \|v\|_{\mathscr{Y}_{T}}:=\|v\|_{C([0,T];H_n)}+\|v\|_{{L_{x,T}^8}}$$
		and
		$$
		\|u\|_{\mathscr{X}^k_T}=\|u\|_{\mathscr{X}_T}+\sum_{i=1}^{k}\|\partial_{x}^i u\|_{\mathscr{X}_T},\ \text{for}\ k\in\mathbb{N^+}.
		$$
		The Sobolev embedding theorem implies $ f\in \mathscr{X}^k_T, \ \forall k\in \mathbb{N^+}$.
		Consider the equivalent integral equation:
		\begin{equation*}
			\begin{aligned}
				&u(t) = S(t)u_0\\
				&\qquad\quad+\mathscr{A}((u+f)( \gamma_1 \psi_K(|u+f|^2)(v+g)+\beta \varphi_K (|u+f|^2)|u+f|^2))\\
				&v(t) = U(t)v_0\\
				&\qquad\quad+\mathscr{B}(\gamma_2\varphi_K(|u+f|^2)|u+f|^2-\varphi_K(v+g)(v+g)^2/2),
			\end{aligned}
		\end{equation*}
		which introduce a map $\mathscr{J}$ from $\mathscr{X}_T\times \mathscr{Y}_{T}$ to $\mathscr{X}_T\times \mathscr{Y}_{T}$:
		\begin{equation*}
			\begin{aligned}
				&\mathscr{J}u(t) = S(t)u_0\\
				&\qquad\quad+\mathscr{A}((u+f)( \gamma_1 \psi_K(|u+f|^2)(v+g)+\beta \varphi_K (|u+f|^2)|u+f|^2)),\\
				&\mathscr{J}v(t) = U(t)v_0\\
				&\qquad\quad+\mathscr{B}(\gamma_2\varphi_K(|u+f|^2)|u+f|^2-\varphi_K(v+g)(v+g)^2/2).
			\end{aligned}
		\end{equation*}
		Here $\mathscr{A}(F)=-i\int_0^tS(t-t')F(t')~dt'$, $\mathscr{B}(G) = \int_0^t U(t-t')P_n\partial_xG(t')~dt'$.


		By Strichartz estimates, we have $\|S(t)u_0\|_{\mathscr{X}_T}\lesssim \|u_0\|_{L^2}$, $\|U(t)v_0\|_{\mathscr{Y}_T}\lesssim \|v_0\|_{L^2}$ and
		\begin{align*}
			\|\mathscr{A}(F)\|_{\mathscr{X}_T}\leq C\|F\|_{L_{x,T}^{6/5}},\quad \|\mathscr{B}(G)\|_{\mathscr{Y}_T}\leq nC\|G\|_{L_{x,T}^{8/7}},
		\end{align*}
		where $C$ is uniform for any $T>0$.
		
		To do a fixed point argument, we prove following estimates.
		By interpolation and H\"{o}lder inequality, we have
		\begin{align}
			&\begin{aligned}\label{nonlinearesti1}
				\|(u+f)\psi_K(|u+f|^2)(v+g)\|_{L_{x,T}^{6/5}}
				&\leq T^{\frac{3}{8}}\|u+f\|^\frac{3}{4}_{L_t^\infty L_x^2}\|u+f\|^\frac{1}{4}_{L_{x,T}^6}\\
				&\quad \cdot T^{\frac{7}{18}} \|v+g\|^{\frac{7}{9}}_{L_t^\infty L_x^2}\|v+g\|^{\frac{2}{9}}_{L_{x,T}^8},
			\end{aligned}\\
			&\|(u+f)\varphi_K(|u+f|^2)|u+f|^2\|_{L_{x,T}^{6/5}}\leq T^{\frac{1}{3}}\|u+f\|_{L_t^\infty L_x^2}\|u+f\|^2_{L_{x,T}^6},\label{nonlinearesti12}\\
			&\|(u+f)^2\varphi_K(|u+f|^2)\|_{L_{x,T}^{8/7}}\leq T^{\frac{13}{16}}\|u+f\|^{\frac{13}{8}}_{L_t^\infty L_x^2}\|u+f\|^{\frac{3}{8}}_{L_{x,T}^6},\label{nonlinearesti3}\\
			&\|(v+g)^2\varphi_K(v+g)\|_{L_{x,T}^{8/7}}\leq T^{\frac{5}{6}}\|v+g\|^{\frac{5}{3}}_{L_t^\infty L_x^2}\|v+g\|^{\frac{1}{3}}_{L_{x,T}^8}.\label{nonlinearesti4}
		\end{align}
		Thus, by  \eqref{nonlinearesti1}-\eqref{nonlinearesti4}, one has 
		\begin{align}
			&\begin{aligned}\label{boundedesti1}
				\|u(t)\|_{\mathscr{X}_T}&\leq\| u_0 \|_{L_x^2}+ \| S(t)u_0 \|_{L_{x,T}^{6}}\\
				&\quad+C(T^{\frac{55}{72}}\|u+f\|_{\mathscr{X}_T}\|v+g\|_{\mathscr{Y}_T}+T^{\frac{1}{3}}\|u+f\|^3_{\mathscr{X}_T}),
			\end{aligned}\\
			&\begin{aligned}\label{boundedesti2}
				\|v(t)\|_{\mathscr{Y}_{T}}& \leq\| v_0 \|_{L_x^2}+ \| S(t)v_0 \|_{L_{x,T}^{8}}\\
				&\quad+nC(T^{\frac{13}{16}}\|u+f\|^2_{\mathscr{X}_T}+T^{\frac{5}{6}}\|v+g\|^2_{\mathscr{Y}_T}).
			\end{aligned}
		\end{align}
		Also,
		\begin{equation*}
			\begin{aligned}
				\|u_1-u_2\|_{\mathscr{X}_T}&\leq CT^{\frac{55}{72}}( \|u_1-u_2\|_{\mathscr{X}_T}\|v_2+g\|_{\mathscr{Y}_T}+\|v_1-v_2\|_{\mathscr{Y}_T}\|u_1+f\|_{\mathscr{X}_T})\\
				&\quad+2C\|(u_1+f)(v_1+g)|u_1-u_2|(|u_1+f|+|u_2+f|)\|_{L_{x,T}^{6/5}}\\
				&\quad+C\||u_1-u_2|(|u_1+f|+|u_2+f|)|u_1+f|^2(u_1+f)\|_{L_{x,T}^{6/5}}\\
				&\quad+C\|(|u_1+f|+|u_2+f|)|u_1-u_2|(u_2+f)\|_{L_{x,T}^{6/5}}\\
				&\quad+C\||u_1+f|^2(u_1-u_2)\|_{L_{x,T}^{6/5}}.
			\end{aligned}
		\end{equation*}
		Thus,
		\begin{equation}\label{constractingesti1}
			\begin{aligned}
				\|u_1-u_2\|_{\mathscr{X}_T}
				&\leq CT^{\frac{55}{72}}( \|u_1-u_2\|_{\mathscr{X}_T}\|v_2+g\|_{\mathscr{Y}_T}+\|v_1-v_2\|_{\mathscr{Y}_T}\|u_1+f\|_{\mathscr{X}_T})\\
				&\quad+2CT^\frac{5}{16}\|u_1+f\|^2_{\mathscr{X}_T}\|v_1+g\|_{\mathscr{Y}_T}\|u_1-u_2\|_{\mathscr{X}_T}\\
				&\quad+2CT^\frac{5}{16}\|u_1+f\|_{\mathscr{X}_T}\|u_2+f\|_{\mathscr{X}_T}\|v_1+g\|_{\mathscr{Y}_T}\|u_1-u_2\|_{\mathscr{X}_T}\\
				&\quad+C\|u_1-u_2\|_{\mathscr{X}_T}\|u_1+f\|^3_{L^6_{x,T}} ( \|u_1+f\|_{\mathscr{X}_T}+\|u_2+f\|_{\mathscr{X}_T} ) \\
				&\quad+CT^\frac{1}{2}(\|u_1+f\|_{\mathscr{X}_T}+\|u_2+f\|_{\mathscr{X}_T}) \|u_1-u_2\|_{\mathscr{X}_T}\|u_2+f\|_{\mathscr{X}_T}  \\
				&\quad+CT^\frac{1}{2}\|u_1-u_2\|_{\mathscr{X}_T}\|u_1+f\|^2_{\mathscr{X}_T}.
			\end{aligned}
		\end{equation}
		We point out that the third term in the right of \eqref{constractingesti1}, which come from cubic term, will bring some extra difficulties because of the lack of $T$.
		
		In the same manner, we have
		\begin{equation}\label{constractingesti2}
			\begin{aligned}
				\|v_1-v_2\|_{\mathscr{Y}_T}&\leq  nT^{\frac{5}{16}}C \|u_1+f\|^2_{\mathscr{X}_T}(\|u_1+f\|_{\mathscr{X}_T}+\|u_2+f\|_{\mathscr{X}_T})\|u_1-u_2\|_{\mathscr{X}_T}\\
				&\quad+nT^{\frac{13}{16}}C\|u_1-u_2\|_{\mathscr{X}_T}(\|u_1+f\|_{\mathscr{X}_T}+\|u_2+f\|_{\mathscr{X}_T}) \\
				&\quad+ nT^{\frac{1}{2}}C\|v_1+g\|^2_{\mathscr{Y}_T}(\|v_1+g\|_{\mathscr{Y}_T}+\|v_2+g\|_{\mathscr{Y}_T})\|v_1-v_2\|_{\mathscr{Y}_T} \\
				&\quad+nT^{\frac{5}{6}}C\|v_1-v_2\|_{\mathscr{Y}_T}(\|v_1+g\|_{\mathscr{Y}_T}+\|v_2+g\|_{\mathscr{Y}_T}).
			\end{aligned}
		\end{equation}
		The $C$ in \eqref{constractingesti1} and \eqref{constractingesti2} is only depend on $\gamma_1,\gamma_2,\beta$.
		
		To make $\mathscr{J}$ contracting, we choose $T>0$ from two aspects. On one hand, for the terms that have factor $T$, we use small property of $T$. On the other hand, we choose $T$ sufficiently small to use the small property of $\|u_1+f\|^3_{L^6_{x,T}} $ for the term without $T$.

		Hence, by a fixed point argument, we obtain the local well-posedness in $L^2\times H_n$. 
		
		To extend the solution, we need a prior estimate for $\|(u,v)\|_{L^2\times L^2}$. Let 
		$$H(t) = \|u(t)\|_{L^2}^2+\|v(t)\|_{L^2}^2.$$
		For $0\leq t\leq T$, we have
		\begin{align*}
			H'(t) &= \int_\mathbb{R}2\mathrm{Re}(\bar{u}u_t)+2vv_t~dx \\
			&= 2\int_{\mathbb{R}}\mathrm{Im}(\gamma_1 \bar{u}f\psi_K(|u+f|^2)(v+g)+\beta \bar{u}f\varphi_K(|u+f|^2)|u+f|^2)\\
			&\quad - (\gamma_2\varphi_K(|u+f|^2)|u+f|^2-\varphi_K(v+g)(v+g)^2/2)P_n\partial_xv~dx\\
			&\leq C_1(\gamma_1,\gamma_2,\beta,k,\|f\|_{L_t^\infty L_x^\infty},\|g\|_{L_t^\infty L_x^2})+C_2(\gamma_1,n,\|f\|_{L_t^\infty L_x^\infty})H(t).
		\end{align*}
		By the Gronwall inequality, we have $$H(t)\leq e^{C_2t}(H(0)+C_1/C_2e^{-C_2T}).$$
		
		However, since the choice of $T$ not only rely on $\|u_0\|_{L_x^2}, \|v_0\|_{L_x^2}$, but also rely on $ \|u_1+f\|_{L^6_{x,T}}$, we can not get the global well-posedness in $\mathscr{X}_T\times\mathscr{Y}_T$ directly from the prior estimate for $\|(u,v)\|_{L^2\times L^2}$. We prove it by a contradiction.
		
		Assume the maximal existence interval of $(u,v)$ is $[0,T)$ with $T<\infty$. 
		By the proof of local well-posedness, we have  
		$$\|u\|_{L_{x,t\in (T',T)}^6}+\|v\|_{L_{x,t\in (T',T)}^8} = \infty, \ \forall~T'<T.$$
		But for any $T'<T''<T$, by the Strichartz estimates, we have
		\begin{align*}
			&\quad\|u\|_{L_{x,t\in (T',T'')}^6}+\|v\|_{L_{x,t\in (T',T'')}^8}\\
			&\leq C(\|u_{T'}\|_{L_x^2}+\|v_{T'}\|_{L_x^2})+  C(K,n,\gamma_1,\gamma_2,\beta)\int_{T'}^{T''}\|u(s)\|_{L^2_x}+\|v(s)\|_{L^2_x}~ds
		\end{align*}
		Hence, we obtain a contradiction, which finishes the proof of global well-posedness in  $\mathscr{X}_T\times\mathscr{Y}_T$. We should note that although there are five power terms in \eqref{constractingesti1}, our model is essentially subcritical.
		
		Finally, we improve the regularity of the solution to $ H^2_x\times H_n$.
		We consider the following inequalities:
		\begin{align*}
			\|u\|_{L^\infty_TH_x^2} & \leq C\|u_0\|_{H_x^2} +CT^{\frac{5}{6}}\|u+f\|_{L^\infty_TH_x^2}(\|v\|_{\mathscr{Y}_T}+\|g\|_{L_T^\infty H_x^3}) \\
			&\quad+CT^\frac{1}{2}\|u+f\|^2_{\mathscr{X}_T} \| u+f\|_{L_T^\infty H_x^2}+T^{\frac{1}{2}}C\|u+f\|^3_{\mathscr{X}_T},
		\end{align*}
		where $C$ depends on $ n, k,\gamma_1,\gamma_2,\beta $.
		Thus, we can get $u\in C([0,T];H^2_x)$ step by step, which finishes the proof.
	\end{proof} 
	
	\section{The Convergence of Approximation Equations}\label{chapterconvergence}
	
	In this section, to get the priori estimate of $ \|(u,v)\|_{L_T^\infty \mathcal{H}_x^1} $, we show the convergences of approximation equations in the sense of  $X_1(T)$  $a.s. \mathbb{P}$. 
	
	For any fixed $T>0$ and almost surely $\omega\in\Omega $, according to Proposition \ref{prop:linear} and Proposition \ref{prop:globalwellposedness}, we have 
	$$\int_{0}^{t}S(t-s)\Phi_m^{(1)}dW_t^{(1)},\int_{0}^{t}U(t-s)\Phi_m^{(2)}dW_t^{(2)}\in C([0,T];H_x^1)$$
	and $
	\sup_{K\in \mathbb{N^+}}\|u_{m,n,K}\|_{C([0,T];H_x^1)}<\infty, \ \sup_{K\in \mathbb{N^+}}\|v_{m,n,K}\|_{C([0,T];H_x^1)}<\infty$ 
	almost surely.
	Let
	\begin{align*}
		M(\omega)&=\left\|\int_{0}^{t}S(t-s)\Phi_m^{(1)}dW_t^{(1)}\right\|_{C([0,T];H_x^1)}+ \sup_{K\in \mathbb{N^+}}\|u_{m,n,K}\|_{C([0,T];H_x^1)}\\
		&\quad+\left\|\int_{0}^{t}U(t-s)\Phi_m^{(2)}dW_t^{(2)}\right\|_{C([0,T];H_x^1)}+\sup_{K\in \mathbb{N^+}}\|v_{m,n,K}\|_{C([0,T];H_x^1)}.
	\end{align*}
	Thus, as long as we choose $ K> C (M^2(\omega)+M(\omega))$ ($C$ only depends on the support of $\varphi$), \eqref{approxila} is equivalent to
	\begin{equation}\label{middleapproxila}
		\left\{
		\begin{aligned}
			&d u_{m,n}=i\partial_{xx}u_{m,n}dt -i(\gamma_1 u_{m,n}v_{m,n}+\beta |u_{m,n}|^2u_{m,n})dt + \Phi_m^{(1)} dW_t^{(1)},\\
			&d v_{m,n}=-\partial_{xxx}v_{m,n}dt+  P_n\partial_x(\gamma_2|u_{m,n}|^2-\frac{1}{2}v_{m,n}^2)dt+ \Phi_m^{(2)}dW_t^{(2)},\\
			&u_m(0)=P_mu_0,\ v_m(0)=P_mv_0.
		\end{aligned}
		\right.
	\end{equation}
	Hence, we have 
	\begin{equation}\label{convergein_K}
		\lim_{K\uparrow\infty} (u_{m,n,K},v_{m,n,K}) = (u_{m,n},v_{m,n})  \ \text{in}  \ X_1(T) \  a.s.\mathbb{P}.
	\end{equation} 
	
	In the next proposition, we will consider the convergence  of \eqref{middleapproxila} as $n\uparrow\infty$. Since the limiting equations are not cut off in the Fourier space any more, we use the workspace $X_1(T)$.
	
	\begin{prop}\label{prop:converge_n}
		For any $T>0$, we have 
		\begin{equation}\label{convergence_n}
			\lim_{n\uparrow\infty} (u_{m,n},v_{m,n}) = (u_{m},v_{m})  \ \text{in}\  X_1(T) \  a.s.\mathbb{P}.
		\end{equation}
	\end{prop}
	\begin{proof}
		Firstly, we note that the local well-posedness of \eqref{smoothinitialandnoise} can be proved easily.
		By the priori estimate \eqref{priori_esti_finish} of $(u_{m,n},v_{m,n}) $ in $C([0,T];\mathcal{H}_x^1)$ and \eqref{convergein_K}, we have
		$$ 
		M_1(\omega) :=\sup_{n,K\in\mathbb{N^+}}\|(u_{m,n,K},v_{m,n,K})\|_{L_T^\infty \mathcal{H}_x^1} < \infty. 
		$$		
		To show $( u_{m,n},v_{m,n} )$ converging to $ (u_m,v_m) $, let $$(\tilde{u}_{m,n},\tilde{v}_{m,n})=(u_m,v_m)-( u_{m,n},v_{m,n} ). $$
		Then $ \tilde{u}_{m,n},\tilde{v}_{m,n} $ satisfy the equations
		\begin{equation}\label{converge_equation_u}
			\begin{aligned}
				d\tilde{u}_{m,n}-i\partial_{xx}\tilde{u}_{m,n}dt&
				=-i\gamma_1(u_{m}\tilde{v}_{m,n}+\tilde{u}_{m,n}v_{m,n})dt\\
				&\quad-i\beta[|u_{m}|^2\tilde{u}_{m,n}+u_{m}u_{m,n}\bar{\tilde{u}}_{m,n}+|u_{m,n}|^2\tilde{u}_{m,n}]dt.
			\end{aligned}
		\end{equation}
		and
		\begin{equation}\label{converge_equation_v}
			\begin{aligned}
				d\tilde{v}_{m,n}+\partial_{xxx} \tilde{v}_{m,n}dt
				&=\partial_{x}P_n[\gamma_2(u_m\bar{\tilde{u}}_{m,n}+\bar{u}_{m,n}\tilde{u}_{m,n})-\frac{1}{2}\tilde{v}_{m,n}(v_{m}+v_{m,n})]dt\\
				&\quad+(I-P_n)\partial_{x}(\gamma_2|u_{m}|^2-\frac{1}{2}v_{m}^2)dt
			\end{aligned}
		\end{equation}
		with initial value $ (\tilde{u}_{m,n}(0),\tilde{v}_{m,n}(0))=(0,0)$.
		
		Let $T^*(\omega)\in(0,T]$ such that $u_m\in X^1_{T^*}$, $v_m\in X^2_{T^*}$.
		We have 
		\begin{equation}\label{tilde_u_mn}
			\|\tilde{u}_{m,n}\|_{X^1_1(t^*)} \leq t^*C(\gamma_1,\beta,M_1,\|(u_m,v_m)\|_{X_1(T^*)})\|(\tilde{u}_{m,n},\tilde{v}_{m,n})\|_{X_1(t^*)}
		\end{equation}
		and
		\begin{equation}\label{tilde_v_mn}
			\begin{aligned}
				\|\tilde{v}_{m,n}\|_{X^2_1(t^*)} &\leq(t^*)^{\frac{1}{2}}C(\gamma_2,M_1,\|(u_m,v_m)\|_{X_1(T^*)})\|(\tilde{u}_{m,n},\tilde{v}_{m,n})\|_{X_1(t^*)}\\
				&\quad+\|\mathscr{B}(I-P_n)\partial_{x}(\gamma_2|u_m|^2-\frac{1}{2}v_m^2)\|_{X_1^2(t^*)},
			\end{aligned}
		\end{equation}
		for any $t^*\in[0,T^*]$.
		
		By \eqref{tilde_u_mn}, \eqref{tilde_v_mn} and choosing $t^*$ sufficiently small (independent of $n$), we obtain
		\begin{align*}
			\|(\tilde{u}_{m,n},\tilde{v}_{m,n})\|_{X_1(t^*)}\leq 2 \|\mathscr{B}(I-P_n)\partial_x(\gamma_2|u_m|^2-\frac{1}{2}(v_m)^2)\|_{X_1^2(t^*)}.
		\end{align*}
		To show the righthand tend to $0$ when $n\uparrow \infty$, we need a modification of $I-P_n$. Let $\tilde{P}_{>n}:=\mathscr{F}^{-1} \tilde{\varphi}(\xi/n)\mathscr{F}$ where $\tilde{\varphi}\in C^\infty(\mathbb{R})$, $\tilde{\varphi}|_{|\xi|\geq 1} = 1$, $\tilde{\varphi}|_{|\xi|\leq 1/2} =0 $. Then, $I-P_n = \tilde{P}_{>n} (I-P_n)$. 
		By the group estimate of $U(t)$ and Lemma \ref{lem:converge_n}, we  have
		\begin{equation}\label{converge_to_0_n}
			\begin{aligned}
				&\quad\lim_{n\uparrow\infty}\|\mathscr{B}(I-P_n)\partial_x(\gamma_2|u_m|^2-\frac{1}{2}(v_m)^2)\|_{X_1^2(t^*)}\\
				&=\lim_{n\uparrow\infty}\|\mathscr{B}\tilde{P}_{>n} (I-P_n)\partial_x(\gamma_2|u_m|^2-\frac{1}{2}(v_m)^2)\|_{X_1^2(t^*)}\\
				&\leq \lim_{n\uparrow\infty} C(T) \|\tilde{P}_{>n}J^1(\gamma_2|u_m|^2-\frac{1}{2}(v_m)^2)\|_{L_x^1 L_{t^*}^2}\\
				&=0,
			\end{aligned}
		\end{equation}
		which implies that 
		\begin{equation}\label{converge_interval}
			\lim_{n\uparrow\infty} \|(u_{m,n}-u_{m},v_{m,n}-v_m)\|_{X_1(t^*)}=0.
		\end{equation}
	 Furthermore, we divide  $[0,T]$ into finite small intervals $[0,t^*], \ [t^*,2t^*]\cdots$. Hence, we can extend \eqref{converge_interval} to 
		$$\lim_{n\uparrow\infty} \|(u_{m,n}-u_{m},v_{m,n}-v_m)\|_{X_1(T)}=0,$$
		by the above method and $\lim_{n\uparrow\infty}\|(\tilde{u}_{m,n}(t^*),\tilde{v}_{m,n}(t^*))\|_{\mathcal{H}_x^1}=0$ step by step.
	\end{proof}
	
	\begin{lemma}\label{lem:converge_n}
		Using the notations of Proposition \ref{prop:converge_n}, for any $f\in L_x^1L_T^2$, we have 
		$$
		\lim_{n\uparrow\infty} \|\tilde{P}_{>n} f\|_{L_x^1L_T^2}=0.
		$$
	\end{lemma}
	\begin{proof}
		First, we can find a sequence of $\{f_n\}_{n\in\mathbb{N^+}} \subset \mathscr{S}(\mathbb{R}\times[0,T])$, such that $\lim_{k\uparrow\infty} f_k=f $ in sense of $L_x^1L_T^2$. Since $\tilde{\varphi}(\xi/n)-1\in C_0^{\infty}(\mathbb{R}) $, we have
		\begin{align*}
			&\quad\sup_{n\in \mathbb{N^+}}\|\tilde{P}_{>n} f\|_{L_x^1L_T^2}\\
			&\leq\sup_{n\in \mathbb{N^+}}\|\int_{\mathbb{R}}  n\mathscr{F}^{-1}_{\xi} (\tilde{\varphi}-1)(ny)\cdot f(x-y,t)~dy\|_{L_x^1L_T^2}+ \| f(x,t) \|_{L_x^1L_T^2}\\
			&\leq \|\mathscr{F}^{-1}_{\xi} (\tilde{\varphi}-1)\|_{L_x^1} \| f(x,t) \|_{L_x^1L_T^2}+ \| f(x,t) \|_{L_x^1L_T^2}\\
			&\leq C\| f(x,t) \|_{L_x^1L_T^2},
		\end{align*}
		where $C$ is independent of $ n$. 
		Also, for any $f_k$, by the H\"{o}lder inequality, dominated convergence theorem and monotone convergence theorem we have
		\begin{align*}
			&\quad\lim_{n\uparrow\infty}\|\tilde{P}_{>n} f_k\|_{L_x^1L_T^2}\\
			&\lesssim\lim_{n\uparrow\infty}\|(1+x^2)^{\frac{1}{2}}\tilde{P}_{>n} f_k\|_{L_T^2L_x^2}\\
			&\lesssim\lim_{n\uparrow\infty}\|\tilde{P}_{>n} f_k\|_{L_T^2L_x^2}+ C\lim_{n\uparrow\infty}\|x\tilde{P}_{>n} f_k\|_{L_T^2L_x^2}\\
			&\lesssim\lim_{n\uparrow\infty}\|\tilde{\varphi}(\frac{\xi}{n}) \hat{f}_k(\xi)\|_{L_T^2L_\xi^2}+C\lim_{n\uparrow\infty}\|\partial_{\xi}(\tilde{\varphi}(\frac{\xi}{n}) \hat{f}_k(\xi))\|_{L_T^2L_\xi^2}\\
			&\lesssim\lim_{n\uparrow\infty}\|\frac{1}{n}\tilde{\varphi}'(\frac{\xi}{n}) \hat{f}_k(\xi)\|_{L_T^2L_\xi^2}+C\lim_{n\uparrow\infty}\|\tilde{\varphi}(\frac{\xi}{n} )\hat{f}'_k(\xi)\|_{L_T^2L_\xi^2}\\
			&=0.
		\end{align*}
		
		Thus, we have
		\begin{align*}
			\lim_{n\uparrow\infty} \|\tilde{P}_{>n} f\|_{L_x^1L_T^2}&\leq\lim_{n\uparrow\infty} \|\tilde{P}_{>n} (f-f_k)\|_{L_x^1L_T^2}+\lim_{n\uparrow\infty} \|\tilde{P}_{>n} f_k\|_{L_x^1L_T^2}\\
			&\lesssim \|f-f_k\|_{L_x^1L_T^2}\rightarrow 0,~k\rightarrow \infty,
		\end{align*}
		which finishes the proof.
	\end{proof}
	\begin{lemma}\label{lem:converge_m}
		For any $T>0$, we have
		$$ 
		\lim_{m\uparrow\infty}(u_m,v_m)=(u,v) \ \text{in} \ X_1(T) \ a.s.\mathbb{P}.
		$$
	\end{lemma}
	\begin{proof}
		From \eqref{priori_esti_finish} and the Fatuo Lemma, we have
		\begin{align*}
			\mathbb{E}\sup_{m\in\mathbb{N^+}}\|u_m\|_{L_T^\infty H_x^1}&=\mathbb{E}\sup_{m\in\mathbb{N^+}}\|\lim_{n\uparrow\infty}u_{m,n}\|_{L_T^\infty H_x^1}=\mathbb{E}\sup_{m\in\mathbb{N^+}}\lim_{n\uparrow\infty}\|u_{m,n}\|_{L_T^\infty H_x^1}\\
			&\leq\mathbb{E}\varliminf_{n\uparrow\infty}\sup_{m\in\mathbb{N^+}}\|u_{m,n}\|_{L_T^\infty H_x^1}\\
			&\leq\varliminf_{n\uparrow\infty}\mathbb{E}\sup_{m\in\mathbb{N^+}}\|u_{m,n}\|_{L_T^\infty H_x^1}<\infty.
		\end{align*}
		Thus, we can let
		\begin{equation}\label{final_proof_1}
			M_2(\omega) :=\sup_{m\in\mathbb{N^+}} \|(u_m,v_m)\|_{L_T^\infty \mathcal{H}_x^1}<\infty \ a.s.\mathbb{P}. 
		\end{equation}
		We denote by $[0,T(\omega)]$ the local existence interval of \eqref{smodel} and consider equations that $\{(\tilde{u}_m, \tilde{v}_m)\}_{m\in\mathbb{N^+}}:=\{(u-u_m, v-v_m)\}_{m\in\mathbb{N^+}}$ satisfy:
		\begin{equation}\label{converge_m}
			\left\{
			\begin{array}{ccl}
				d\tilde{u}_m-i\tilde{u}_{m}dt&=&-i\gamma_1(\tilde{u}_mv+u_m\tilde{v}_m)dt-i\beta |u_m|^2\tilde{u}_mdt\\
				&&-i\beta u^2\bar{\tilde{u}}_mdt-i\beta u\bar{u}_m\tilde{u}_mdt+(I-P_m)\Phi^{(1)} dW_t^{(1)},\\
				d\tilde{v}_m+\partial_{xxx}\tilde{v}_{m}dt&=&\partial_{x}(\gamma_2(u\bar{\tilde{u}}_m+\bar{u}_m\tilde{u}_m)-\frac{1}{2}\tilde{v}_m(v+v_m))dt\\
				&&+(I-P_m)\Phi^{(2)} dW_t^{(2)},\\
				(\tilde{u}_m,\tilde{v}_m)&=& ((I-P_m)u_0,\ (I-P_m)v_0).
			\end{array}
			\right.
		\end{equation}
		We note that 
		$$
		\lim_{m\uparrow\infty}\mathbb{E}\left\|(S(t)(I-P_m)u_0,U(t)(I-P_m)v_0)\right\|^2_{X_1(T)}=0
		$$
		and
		\begin{align*}
		&\quad\mathbb{E}\left\|\int_{0}^{t}S(t-s)(I-P_m)\Phi^{(1)} dW_s^{(1)},U(t-s)(I-P_m)\Phi^{(2)} dW_s^{(2)}\right\|^2_{X_1(T)}\\
		&\rightarrow 0,\quad m\rightarrow \infty.
		\end{align*}
		Then, we can choose a subsequence such that linear terms $\{(l_m^1,l_m^2)\}_{m\in\mathbb{N^+}}$ and stochastic force terms $\{(f_m^1,f_m^2)\}_{m\in\mathbb{N^+}}$ converge to $0$ almost surely, as $m\uparrow\infty$. 
		
		Thus, we have
		\begin{align*}
			\|\tilde{u}_m\|_{X_1^1(t^*)}&\leq   C(\gamma_1,\beta,M_2,\|(u,v)\|_{X_{1}(T(\omega))})t^*\|(\tilde{u}_m,\tilde{v}_m)\|_{X_1(t^*)}\\
			&\quad +C(T)\|l_m^1 \|_{X^1_1(T)}+\|f^1_m \|_{X_1^1(T)},
		\end{align*}
		and
		\begin{align*}
			\|\tilde{v}_m\|_{X_1^2(t^*)}&\leq C(\gamma_2,M_2,\|(u,v)\|_{X_{1}(T(\omega))})(t^*)^{\frac{1}{2}}\|(\tilde{u}_m,\tilde{v}_m)\|_{X_1(t^*)}\\
			&\quad  +C(T)\|l_m^2 \|_{X^2_1(T)}+\|f^2_m \|_{X_1^2(T)}
		\end{align*}
		for any $t^*\in(0,T(\omega)]$, which implies that for a sufficiently small $t^*>0$, $\|(\tilde{u}_m,\tilde{v}_m)\|_{X_1(t^*)}$ converges to $0$ in the sense of $X_1(t^*)$ as $m\uparrow \infty$. 
	    Hence, we can get
	    \begin{equation}\label{final_proof_2}
	    	\lim_{m\uparrow\infty} \|(\tilde{u}_m,\tilde{v}_m)\|_{X_1(T)}=0 \ a.s.\mathbb{P},
	    \end{equation}
		by the former arguments and $\lim_{n\uparrow\infty}\|(\tilde{u}_{m}(t^*),\tilde{v}_{m}(t^*))\|_{\mathcal{H}_x^1}=0$ step by step.
	\end{proof}
	
	By the priori estimate and the convergence properties of $u_{m,n,k}$, we can finish the proof of Theorem \ref{thm:main}.
	
	\begin{proof}[\textbf{Proof of Theorem \ref{thm:main}}]		
		Combining \eqref{final_proof_1} and \eqref{final_proof_2}, we have 
		$$
		\|(u,v)\|_{L_T^\infty \mathcal{H}_x^1}<\infty \ a.s.\mathbb{P}. $$
		Hence, we can extend the local solution step by step through fixed point arguments and get a strong solution of \eqref{smodel} almost surely.
	\end{proof}

	According to the Theorem \ref{thm:main}, we can get Corollary \ref{cor:higher_case} for higher regularity cases.
\begin{coro}\label{cor:higher_case}
	Suppose the conditions of Theorem \ref{thm:main} hold. For any $ s\geq 1$, if $\Phi\in L_2^{0,s}\times L_2^{0,s}$ and $(u_0,v_0)\in L^2(\Omega;\mathcal{H}^s)$, then for any $ T>0$, \eqref{smodel} exists a unique strong solution $w \in {X_{s}(T)} \ a.s.\mathbb{P}$. 
\end{coro}
\begin{proof}
	This proof can be followed by Lemma \ref{lem:workspace} and Proposition \ref{prop:linear}.
\end{proof}	

	\phantomsection
	\bibliographystyle{amsplain}
	\addcontentsline{toc}{section}{References}
	\bibliography{reference}

	\scriptsize\textsc{Jie Chen: Institute of Applied Physics and Computational Mathematics, Beijing 100088, P.R. China.}
	
	\textit{E-mail address}: \textbf{jiechern@163.com}
	\vspace{20pt}
	
	\scriptsize\textsc{$*$ Corresponding Author}
	
	\scriptsize\textsc{Fan Gu: Institute of Applied Physics and Computational Mathematics, Beijing 100088, P.R. China.}
	
	\textit{E-mail address}: \textbf{gufan@amss.ac.cn}
	\vspace{20pt}
	
	\scriptsize\textsc{Boling Guo: Institute of Applied Physics and Computational Mathematics, Beijing 100088, P.R. China.}
	
	\textit{E-mail address}: \textbf{gbl@mail.iapcm.ac.cn}
	\vspace{20pt}
\end{document}